\documentclass[a4paper,11pt]{amsart}
\usepackage{amsmath,amssymb,amsfonts,phonetic}

\usepackage{epsfig}

\usepackage{mathrsfs}

\newtheorem{defn}{Definition}[section]

\newtheorem{thm}{Theorem}[section]

\newtheorem{prop}{Proposition}[section]

\newtheorem{rmk}{Remark}[thm]

\newtheorem{lma}{Lemma}[section]

\newtheorem{corl}{Corollary}[section]

\def\N{{\rm I\kern-0.16em N}}

\def\R{{\rm I\kern-0.16em R}}

\def\E{{\rm I\kern-0.16em E}}

\def\P{{\rm I\kern-0.16em P}}

\def\F{{\rm I\kern-0.16em F}}

\def\B{{\rm I\kern-0.16em B}}

\def\C{{\rm I\kern-0.46em C}}

\def\G{{\rm I\kern-0.50em G}}

\numberwithin{equation}{section}

\font\eka=cmex10

\usepackage{ae}

\def\ind{\mathrel{\hbox{\rlap{%

\hbox to 7.5pt{\hrulefill}}\raise6.6pt\hbox{\eka\char'167}}}}

\begin{document}

\title[fBm and continuous functions with bounded variation]{When does fractional Brownian motion not behave as a continuous function with bounded variation?}

\author{Ehsan Azmoodeh, Heikki Tikanm\"aki and Esko Valkeila}

\address{Aalto University, School of Science and Technology, Department  of Mathematics and Systems Analysis, P.O. Box 11100, 00076 AALTO,  Finland}
\email{azmoodeh@cc.hut.fi}

\thanks{Azmoodeh and Tikanm\"aki are grateful to Finnish Graduate School in Stochastic and Statistic, FGSS for the financial support, and Valkeila acknowledges the support of Academy of Finland, grant 212875.}

\smallskip

\begin{abstract}
If we compose a smooth function $g$ with fractional Brownian motion  $B$ with Hurst index $H > \frac12 $, then the resulting change of 
variables formula [or It{\^o}- formula] has the same form as  if fractional Brownian motion would be a continuous function with bounded variation. 
In this note we prove a new integral representation formula for the running maximum of a continuous function with bounded  variation. Moreover we show that the analogue to fractional Brownian motion fails. 
\medskip

\smallskip


\end{abstract}
\maketitle

\section{Introduction}
\subsection{Fractional Brownian motion as a continuous function with bounded variation}
Let $B$ be a fractional Brownian motion with Hurst index $H \in (0,1)$: $B_0=0$, $B$ is continuous centered Gaussian process with covariance function 
$$
\E (B_sB_t) = \frac12\left(t^{2H}+s^{2H} - | t-s| ^{2H}\right) \quad s,t \ge 0.
$$

We fix $T>0$ and work on the interval $[0,T]$. 
We recall that the fractional Brownian motion with Hurst parameter  $H > \frac12 $ has zero quadratic variation:
assume that $\{ \pi _n \}$ is a sequence of partitions of the interval $[0,T]$ such that
\begin{equation*}
\Vert \pi_n \Vert := \max_{1 \le i \le k(n)} (t^{n}_{i} - t^{n}_{i-1} ) \to 0  \quad \text{as} \quad n \to \infty ;
\end{equation*}
then
$$
\lim _{n\to\infty} \sum _{t^n_i \in \pi ^n} \left( B_{t^n_i} - B_{t^n_{i-1}}\right) ^2 \to 0 \quad \mbox{a.s.}
$$
\begin{itemize}
 \item[1.] \ The fact that $B$ has zero quadratic variation allows one to prove the following result. 
Assume that $g\in C_1 (\R )$, and put $g_x = \frac{\partial }{\partial x } g $. Then the following
change of variables formula holds: 
$$
g(B_T) = g(B_0) + \int _0^T g_x(B_s)dB_s ;
$$
here the stochastic integral is a Riemann-Stieltjes integral by the Young- integration theory (see \cite{m} for more details). Note that this change of variables formula is the same for continuous functions with bounded variation.
\item[2.] \ Assume now that $g$ is a convex function, and let $g_x^-$ be the left derivative of $g$. Then we have 
the following change of variables formula:
$$
g(B_T) = g(B_0) + \int _0^T g_x^-(B_s)dB_s .
$$
Here the integral is a generalized Lebesgue--Stieltjes integral. One can even show that here the integral is a limit of Riemann sums (see \cite{a-m-v} for more details). Again we have that fractional Brownian motion behaves as it was a continuous function with bounded variation.  
\end{itemize}

In the above two cases fractional Brownian motion behaves as a continuous function with bounded variation.
So it is natural to ask, how far this similarity goes?  We will prove an integral representation for the running maximum of a continuous function with bounded variation.  It turns out that here the analogy between fractional Brownian motion and a continuous function with bounded variation ends. More precisely, the corresponding formula does not hold for fractional Brownian motion in the sense of generalized Lebesgue-Stieltjes integral nor as a limit of Riemann-Stieltjes sums.

\subsection{The stochastic integral representation of the maximum of Brownian motion}

Let $W = \{ W_t \}_{ t \in [0,T]}$ be a standard Brownian motion on the interval $[0,T]$ with natural filtration $\mathcal{F}_{t}^{W}$. A classical result implies that any square integrable random variable $X$, measurable with respect to $\mathcal{F}_{T}^{W}$ admits a stochastic integral representation

\begin{equation*}
X = \E (X) + \int_{0}^{T} H_t dW_t
\end{equation*}
for some $\mathcal{F}_{t}^{W}$-predictable process $H$ $($see \cite{r-y} for more details$)$. The \textit{Clark-Ocone} formula gives an explicit form of the integrand process $H$ in terms of Malliavin derivative, when the random variable $X$ is smooth enough. In the next theorem the notation $ \mathbb{D}^{1,2}$ stands for Hilbert space of random variables with square integrable Malliavin derivative $($see \cite{nu} for more details$)$.

\begin{thm}
Let $X \in \mathbb{D}^{1,2}$. Then $X$ admits the following representation
\begin{equation*}
X = \E (X) + \int_{0}^{T} \E (D_t X | \mathcal{F}_{t}^{W}) d W_t.
\end{equation*}
\end{thm}
See \cite{nu}. The computation of the conditional expectation in the  representation above is sometimes rather difficult in general, but it is possible to handle it in some cases as it is shown below.\\

We denote the maximum random variable of Brownian motion $W$ by $S_T$, i.e. 
\begin{equation*}
S_T = \max _{t \in [0,T]} W_t. 
\end{equation*}
For $S_T$, we have the following result.

\begin{thm} \label{thm:maxbm}
For the random variable $S_T$ the stochastic integral representation
\begin{equation*}
S_T = \E (S_T) + 2 \int_{0}^{T} \Big[ 1- \Phi (\frac{S_t - W_t}{\sqrt{T-t}}) \Big] dW_t
\end{equation*}
holds, where $S_t = \max_{u\le t } W_u$,
\begin{equation*}
\E (S_T)= \E (|W_T|) = \sqrt{\frac{2T}{\pi}}, \qquad \Phi(x) = \frac{1}{\sqrt{2\pi}} \int_{- \infty}^{x} e^{\frac{-y^2}{2}}dy.
\end{equation*}
\end{thm}

\begin{proof}
See \cite{r-w}. For a different proof see \cite{s-y}.
\end{proof}

\subsection{The problem}

Assume $f:[0,T] \to \mathbb{R}$ be a bounded variation function. We denote by $\mu_f$, the signed measure induced by bounded variation function $f$. We are interested in whether the following representation

\begin{equation}\label{eq:fun}
f^{*}(T) = f(0) + \int _{0}^{T} \textbf{1}_{ \{ f^{*}(t) = f(t)\} } df(t) = f(0) + \int _{0}^{T} \textbf{1}_{ \{ f^{*}(t) = f(t)\} } d\mu_{f}(t).
\end{equation}

holds, where $f^{*}$ is the running maximum function i.e.

\begin{equation*}
 f^{*}(t):= \max_{0 \le s \le t } f(s).
\end{equation*}

The integral in the right hand side (\ref{eq:fun}) is understood in the \textit{Lebesgue-Stieltjes} integral sense. We will show that equation $($\ref{eq:fun}$)$ holds for continuous bounded variation functions but not for fractional Brownian motion.

\section{Auxiliary results}

\subsection {Facts on bounded variation functions}

We recall some results on bounded variation functions. First, recall that for every bounded variation function $f:[0,T] \to \mathbb{R}$, the derivative $f^{'}$ exists a.e..

\vskip0.25cm

\begin{thm}Let $f:[0,T] \to \mathbb{R}$ be a $($continuous$)$ bounded variation function. Then

\begin{enumerate}
 \item The function $f$ can be decomposed to difference of two increasing $($continuous$)$ functions i.e.

\begin{equation*} \label{eq:jordan}
 f = V_f - W_f
\end{equation*}

where 

\begin{eqnarray*}&&
 V_f (t) := \sup  \sum _{t_i \in \pi} ( f(t_i) - f(t_{i-1}))^+\\&&
W_f (t)= \sup  \sum _{t_i \in \pi} ( f(t_i) - f(t_{i-1}))^-
\end{eqnarray*}
and the supremum is taken over all partitions $\pi$ of $[0,t]$. Therefore, the Lebesgue-Stieltjes measure $\mu_f$ induced by $($continuous$)$ bounded variation function $f$ can be expressed as the difference of two $($atomless$)$ positive measures ${(\mu_f)}^+$ and ${(\mu_f)}^-$, i.e. $\mu_f = {(\mu_f)}^+ - {(\mu_f)}^-$. Moreover
\begin{equation*}
|\mu_f| = {(\mu_f)}^+ + {(\mu_f)}^-,
\end{equation*}
where $|\mu_f|$ stands for the total variation measure of $\mu_f$.

\item The Lebesgue-Stieltjes measure $\mu_f$ associated to a continuous bounded variation function $f$ can be expressed as the sum of two $($atomless$)$ measures $\mu _{ac}$ and $\mu _{sc}$

\begin{equation*} \label{eq:leb}
 \mu_f = \mu _{ac} + \mu _{sc} \qquad \mu_{ac} \ll m \quad \text{and} \quad \mu_{sc} \perp m,
\end{equation*}

where $m$ stands for Lebesgue measure.
\end{enumerate}
\end{thm}

See~\cite{d}.
\begin{thm}(The fundamental theorem of calculus for Lebesgue integral)
The function $f:[0,T] \to \mathbb{R}$ is absolutely continuous iff $f$ is differentiable a.e., $f^{'} \in L^{1}(m)$ and 
\begin{equation*}
 f(t) - f(0) = \int_{[0,t]} f^{'} dm \qquad t \in [0,T].
\end{equation*}

\end{thm}

See~\cite{ye}.

\begin{thm}  \label{thm:smac}

 Let $\mu_f$ be a Lebesgue-Stieltjes signed measure with $\mu_f \ll m$ i.e. $f$ is an absolutely continuous function. Then

\begin{equation*}
 \mu_f (E) = \int_{E} f^{'} dm 
\end{equation*}

for every bounded measurable set $E$.

\end{thm}

See~\cite{ye}.

\vskip0.5cm

\subsection{Pathwise stochastic integration in fractional Besov-type spaces}

Fractional Brownian motion is not a semimartingale, and hence the stochastic integral with respect to fractional Brownian motion $B^H$ is not always defined. We shall give some details of the construction of generalized Lebesgue-Stieltjes integrals in this section. For more information see \cite[Section 2.1.2]{m}.

\begin{defn}

Fix $ 0 <\beta < 1 $.\\

(i) Let $ W^{\beta}_1 =  W^{\beta}_1 ([0,T])$ be the space of real-valued measurable functions $ f :[0,T] \to \mathbb{R}$ such that

\begin{equation*}
\Vert f \Vert _{1,\beta} := \sup _{0 \le s < t \le T} \left( \frac{|f(t) - f(s)|}{(t-s)^\beta} + \int _{s}^{t} \frac{|f(u) - f(s) |}{(u-s)^ {1+\beta }} du \right) < \infty .
\end{equation*}

(ii) Let $W^{\beta}_2 =  W^{\beta}_2 ([0,T])$ be the space of real-valued measurable functions  $ f :[0,T] \to \mathbb{R}$ such that

\begin{equation*}
\Vert f \Vert _{2,\beta} := \int_{0}^{T} \frac{|f(s)|}{s^ \beta} ds + \int_{0}^{T}\int_{0}^{s} \frac{|f(u) - f(s) |}{(u-s)^ {1+\beta }} du ds < \infty .
\end{equation*}

\end{defn}

\begin{rmk}\label{r:rmk1}

The Besov spaces are closely related to the spaces of H{\"o}lder continuous functions. More precisely,
for any $ 0< \epsilon < \beta \wedge (1- \beta)$,
\vskip0.25cm
\begin{center}

$C^{\beta + \epsilon}([0,T]) \subset W^{\beta}_{1} ([0,T])\subset C^{\beta - \epsilon}([0,T]) \quad \text{and} \quad C^{\beta + \epsilon}([0,T]) \subset W^{\beta}_{2} ([0,T]) $.

\end{center}

where $C^{\gamma }([0,T])$ denotes H\"older continuous functions of order $\gamma$.

\end{rmk}

Recall that almost surely the trajectories of $B^H$ for any $T>0$ and any $0<\gamma < H$ belong to $ C^{\gamma }([0,T]) $. This follows from the Kolmogorov continuity theorem. By remark \ref{r:rmk1} we obtain that almost surely the trajectories of $B^H$ for any $T>0$ and any $0<\beta < H$ belong to $ W^{\beta}_1 ([0,T]) $.\\

Denote by $\Gamma$ the Gamma-function. Recall the left-sided Riemann-Liouville fractional integral operator $I^\beta _+$ of order $\beta > 0$:

$$
(I^\beta _{0+} f)(s) = \frac{1}{\Gamma (\beta)} \int _0^sf(u) (s-u)^{\beta -1} du .
$$

The corresponding right-sided fractional integral operator $I^\beta _- $ is defined by

$$
(I^\beta _{t-}f)(s) = \frac{1}{\Gamma (\beta) }\int _s^t f(u) (u-s)^{\beta -1} du .
$$

\begin{rmk}

If $ f \in  W^{\beta}_1 ([0,T])$, then its restriction to $[0,t] \subseteq [0,T]$ belongs to $I^{\beta}_{-}(L_\infty ([0,t]))$. Also, if  $ f \in  W^{\beta}_2 ([0,T])$, then its restriction to $[0,t] \subseteq [0,T]$ belongs to $I^{\beta}_{+}(L_1 ([0,t]))$, where $I^{\beta}_{-}(L_\infty ([0,t])) $ (resp. $I^{\beta}_{+}(L_1 ([0,t])) $) stand for the image of $ L_\infty ([0,t])$ (resp. $ L_1 ([0,t])$) by the fractional Riemann-Liouville operator $I^{\beta}_{-} $ (resp. $ I^{\beta}_{+}$).(For details we refer to \cite{s-k-m}).

\end{rmk}

\begin{defn}

Let $f:[0,T] \to \mathbb{R}$ and $0< \beta < 1$. If $ f \in I^{\beta}_{+}(L_{1} ([0,T]))$(resp. $f \in I^{\beta}_{-}(L_\infty ([0,T]))$ then the Weyl fractional derivatives are defined by

\begin{multline*}
(D^{\beta}_{0+} f)(x)= \frac{1}{\Gamma(1-\beta)} \left( \frac{f(x)}{x^\beta} + \beta \int_{0}^{x}\frac{f(x) - f(y)}{(x-y)^{\beta + 1}}dy \right) \textbf{1} _{(0,T)}(x),\\
\left( \text{resp}.(D^{\beta}_{T^{-}} f)(x)= \frac{1}{\Gamma(1-\beta)} \left( \frac{f(x)}{(T-x)^\beta} + \beta \int_{x}^{T}\frac{f(x) - f(y)}{(y-x)^{\beta + 1}}dy \right) \textbf{1} _{(0,T)}(x)\right).
\end{multline*}

\end{defn}

For a detailed discussion we refer to \cite{s-k-m}. The following proposition clarifies the construction of the stochastic integrals. This approach is by Nualart and  R\u{a}\rc{s}canu.

\begin{prop}\label{pr:n-r}

Let $ f \in  W^{\beta}_2 ([0,T])$, $ g \in W^{1- \beta}_1 ([0,T])$. Then for any $t \in(0,T]$ the Lebesgue integral

\begin{center}

 $\int_{0}^{t} (D^{\beta}_{0+} f)(x) (D^{1- \beta}_{t-} g_{t-} )(x) dx$

\end{center}

exists, and we can define the \textit{ generalized Lebesgue-Stieltjes integral} by

\begin{equation*}
\int_{0}^t f dg := \int_{0}^{t} (D^{\beta}_{0+} f)(x) (D^{1- \beta}_{t-} g_{t-} )(x) dx .
\end{equation*}

\end{prop}

See~\cite{n-r}.

\begin{rmk}\label{rmk:coinside}
It is shown in \cite{z} that if $f \in C^{\gamma }([0,T]) $ and $g \in C^{\mu }([0,T])$ with $ \gamma + \mu > 1$,then the integral $ \int_{0}^{T} f dg $ exists in the sense of proposition \ref{pr:n-r} and coincides with the Riemann-Stieltjes integral.
\end{rmk}

The next theorem is an estimate for  generalized Lebesgue-Stieltjes integral and it can be used for studying the continuity of the integral.

\begin{thm}\label{t:n-r}

Let $ f \in  W^{\beta}_2 ([0,T])$ and  $ g \in W^{1- \beta}_1 ([0,T])$. Then we have the estimate

\begin{equation*}
\left|\int_{0}^t f dg \right| \le \frac{1}{\Gamma (\beta)} \Vert f \Vert _{2,\beta}\Vert g \Vert _{1,1- \beta}   .
\end{equation*}

\end{thm}

See~\cite{n-r}.

\section{Main Results}

\subsection{The case of continuous bounded variation functions}

Let $f:[0,T] \to \mathbb{R}$ be a bounded variation function. Put 
\begin{equation*}
E = \{ t \in [0,T] : f^{*}(t) = f(t)\}.
\end{equation*}

Now we are ready to give a positive answer to our problem in the case, when $f$ is a continuous bounded variation function.

\vskip0.25cm

\begin{thm}\label{thm:main}

 Let $f:[0,T] \to \mathbb{R}$ be a continuous function of bounded variation. Then

\begin{equation}\label{eq:main}
 f^{*}(T) = f(0) + \int _{0}^{T} \textbf{1}_{ \{ f^{*}(t) = f(t)\} } d\mu_f(t).
\end{equation}

\end{thm}

\begin{proof}
Step $1$. Since $f^{*}$ is an increasing function, we have

\begin{equation*}
 \begin{split}
  f^{*}(T) &= f(0) + \int_{(0,T]} d\mu_{f^{*}}(t)\\
&= f(0) + \mu_{f^{*}} (E) + \mu_{f^{*}}(E^{c}).\\ 
 \end{split}
\end{equation*}

Step $2$. We show that

\begin{equation*}
\mu_{f^{*}}(E^{c}) = 0.
\end{equation*}

Since $E^c$ is an open set, without loss of generality we can assume $E^{c}=(a,b) \subset (0,T)$. So

\begin{equation*}
\mu_{f^{*}}(E^{c}) = f^{*}(b) - f^{*} (a).
\end{equation*}

Assume $f^{*}(b) - f^{*} (a) \neq 0 \quad \Longrightarrow \quad f^{*} (a) < f^{*}(b) $. Take $K \in (f^{*} (a), f^{*} (b))$ and set

\begin{equation*}
t_{0} = \inf \{ t > a : f(t)= K \}.
\end{equation*}

Obviously $a < t_{0} < b $, since $f$ is a continuous function and moreover by the definition of $t_0 $ we have that $f(t_0) = f^{*}(t_0)$. So $t_0 \in E $ which is a contradiction.\\

\vskip0.25cm

Step $3$. We show that
\begin{equation*}
\mu_{f^{*}}(E) = \mu_{f}(E).
\end{equation*}
We know that set $E$ is closed and nonempty. Thus, $T^*=\sup\{t \in E\}\in E$. Clearly $f^*(T)=f^*(T^*)=f(T^*)$. It also holds that $(T^*,T]\subset E^c$ and $I=E^c\backslash (T^*,T]$ is open with measure $\mu_f(I)=\mu_f(E^c)-(f(T)-f(T^*))$.
\begin{equation*}
\begin{split}
f(T)-f(0)&=\mu_f (E)+ \mu_f (E^c)\quad \text{and}\\
f^*(T)-f(0)&= \mu_{f^*}(E)+ \mu_{f^*} (E^c)= \mu_{f^*} (E).
\end{split}
\end{equation*}
It follows that
\begin{equation}
\label{difference}
\mu_{f^*} (E) - \mu_f (E) = \mu_f (I).
\end{equation}
We know that $I$ is an open set and thus can be represented as a countable union of disjoint open intervals i.e.
\begin{equation*}
I=\cup_{n=1}^\infty (a_n,b_n).
\end{equation*}
Note that $a_n,b_n\in E$ as boundary points of $E^c$ and
\begin{equation*}
f(a_n)=f^*(a_n)=f^*(b_n)=f(b_n).
\end{equation*}
Now
\begin{equation}
\label{Ec}
\mu_f (I)= \mu_f (\cup_{n=1}^\infty (a_n,b_n))=\sum_{n=1}^\infty f(b_n)-f(a_n)=0.
\end{equation}
Now we deduce from equations~(\ref{difference})~and~(\ref{Ec}) that
\begin{equation*}
\mu_{f^*} (E) = \mu_f(E)
\end{equation*}

\end{proof}

\vskip0.25cm

When $f$ is an absolutely continuous function of bounded variation, one can give a different proof of equation $($\ref{eq:main}$)$. However, we will use the argument of step $2$ of the proof of theorem \ref{thm:main}. First, we need the following simple lemma.

\begin{lma}\label{lma:abcon}
Let $f:[0,T] \to \mathbb{R}$ be an absolutely continuous function of bounded variation. Then $ f^{*}$ is absolutely continuous function.
\end{lma}

\begin{proof}
This follows from $\mu_{f^*} \ll { (\mu_f)}^+ \ll |\mu_f|$. Moreover, $\mu_f \ll m$ if and only if $|\mu_f| \ll m  $.
\end{proof}

\begin{thm}
Let $f:[0,T] \to \mathbb{R}$ be an absolutely continuous function of bounded variation function. Then
\begin{equation*}
f^{*}(T) = f(0) + \int _{0}^{T} \textbf{1}_{ \{ f^{*}(t) = f(t)\} } d\mu_f(t).
\end{equation*}
\end{thm}

\begin{proof}
Put
\begin{align*}
\Lambda _{1} &= \{ t  \in [0,T] : f^{'} \text{ exists at } t \} \\
\Lambda_{2} &= \{ t  \in [0,T] : {(f^{*})}^{'} \text{ exists at } t \}.
\end{align*}

\begin{equation*}
 \begin{split}
  \int _{0}^{T} \textbf{1}_{ \{ f^{*}(t) = f(t)\} } d\mu_f (t) &= \int_{E \cap \Lambda_1} f^{'} dm \qquad \text{Theorem } \ref{thm:smac}\\
& = \int_{E \cap \Lambda_2}  {(f^{*})}^{'} dm \\
& = \int_{[0,T] \cap \Lambda_2} {(f^{*})}^{'}  dm - \int_{[0,T] \cap E^{c} \cap \Lambda_2 } {(f^{*})}^{'} dm \\
&= \int_{[0,T] \cap \Lambda_2} {(f^{*})}^{'}  dm = \int_{[0,T]} {(f^{*})}^{'}  dm \qquad \text{ Step } 2\\
&= f^{*}(T) - f^{*}(0) = f^{*}(T) - f(0) \qquad \text{Lemma } \ref{lma:abcon}.\\
 \end{split}
\end{equation*}

\end{proof}

\begin{rmk}

Note that the continuity assumption is essential. A jump from below the running maximum to a new maximum value would destroy the representation of equation~(\ref{eq:main}).

\end{rmk}

\vskip0.25cm

\subsection{The case of fractional Brownian motion}

Assume $B=\{ B_t \}_{t \in [0,T]}$ be a fractional Brownian motion with Hurst parameter $H \in ( 0,1)$. We denote by $M= \{ M_t \}_{t \in [0,T]}$ the running maximum of fractional Brownian motion, i.e.

\begin{equation*}
M_t := \max_{0  \le s \le t} B_s \qquad t \in [0,T].
\end{equation*}

We start  with the following fact on running maximum.
\begin{lma} \label{lma:low}
For all $t \in (0,T]$,
\begin{equation*}
\mathbb{P} \{ M_t = 0 \} = 0.
\end{equation*}
\end{lma}
\begin{proof}
By contrary assume that there exists a $ t_0 \in (0,T]$ such that $\P (A) >0$, where $A= \{ \omega \in \Omega : \ M_{t_0}(\omega) = 0 \}$. Then 
\begin{equation*}
\limsup_{\epsilon \to 0^+} \frac{B_\epsilon }{\sqrt{2\epsilon^{2H}\log \log (\frac{1}{\epsilon})}} \le 0 \quad \text{for all } \omega \in A.
\end{equation*}

Now, this is a contradiction with the \textit{law of iterated logarithm} for fractional Brownian motion \cite{ar}: for all $t\ge 0$, almost surely
\begin{equation*}
\limsup_{\epsilon \to 0^+} \frac{B_{t+\epsilon} - B_t}{\sqrt{2\epsilon^{2H}\log \log (\frac{1}{\epsilon})}} = 1.
\end{equation*}
\end{proof}

First, we note that the set

\begin{equation}\label{eq:set}
E = \{ (t,\omega) \in [0,T]\times \Omega : M_t (\omega) = B_t (\omega) \}
\end{equation}

is product-measurable: $E \in \mathcal{B}([0,T])\otimes \mathcal{F}$, because the process $B$ is separable. We denote the sections of $E$ by

\begin{equation*}
E_t := \{ \omega \in \Omega : (t,\omega) \in E \}, \qquad E_{\omega}:= \{ t \in [0,T] : (t,\omega) \in E \}.
\end{equation*}

 For $t \in (0,T]$ we have

\begin{equation*}
\mathbb{P} \{ E_t \} = \mathbb{P} \{ B_t - B_s \ge 0 : \forall s \in [0,t] \} = \mathbb{P} \{ B_{t-s} - B_t \le 0 : \forall s \in [0,t] \} = 0
\end{equation*}

because of the fact that the process $\tilde{B} = \{ \tilde{B_s} \} = \{  B_{t-s} - B_t \}_{s \in [0,t]}$ is a fractional Brownian motion and lemma \ref{lma:low}. Therefore by Fubini's theorem we have

\begin{equation*}
 \int_{\Omega} m (E_{\omega}) d\mathbb{P}= (m \times \mathbb{P}) (E) = \int_0^T \mathbb{P}(E_t) dt = 0
\end{equation*}

which implies that 

\begin{equation*}
 m(E_{\omega}) = 0 \quad \text{almost surely}.
\end{equation*}

\begin{thm}\label{thm:sections}
For the set $E$ defined by equation~(\ref{eq:set}) we have
\begin{equation*}
\mathbb{P} \{ E_t \}= 0 \quad \forall t \in (0,T], \qquad m(E_{\omega}) = 0 \quad \text{almost surely}.
\end{equation*}
\end{thm}

Now we have the following theorem.

\begin{thm}
Let $\{ \pi _n \}$ be any sequence of partitions of the interval $[0,T]$ such that $\Vert \pi_n \Vert \to 0  \quad \text{as} \quad n \to \infty$, and 
let $B=\{ B_t \}_{t \in [0,T]}$ be a fractional Brownian motion with Hurst parameter $H \in (0,1)$. Then
\begin{equation*}
\begin{split}
\int_{0}^{T} \textbf{1}_{ \{ B_t = M_t \} } dB_t :&= \lim_{n \to \infty} \sum_{t^{n}_{i} \in \pi_n} \textbf{1}_{ \{ B_{t^{n}_{i-1}} = M_{t^{n}_{i-1}} \} } (B_{t^{n}_{i}} - B_{t^{n}_{i-1}})= 0\\
\end{split}
\end{equation*}

almost surely.

\end{thm}

\begin{proof}

Put $\xi_n := \sum_{t^{n}_{i} \in \pi_n} \textbf{1}_{ \{ B_{t^{n}_{i-1}} = M_{t^{n}_{i-1}} \} } (B_{t^{n}_{i}} - B_{t^{n}_{i-1}})$. Then for every $\epsilon > 0$,

\begin{equation*}
\begin{split}
\mathbb{P} \{  |\xi_n| > \epsilon \} \le \mathbb{P} \{  B_{t^{n}_{i}} = M_{t^{n}_{i}}, \text{for some} \quad 0 < i \le k(n) \} = 0  
\end{split}
\end{equation*}

by theorem $\ref{thm:sections}$.

\end{proof}

\begin{corl}

Let $\{ \pi _n \}$ be any sequence of partitions of the interval $[0,T]$ such that $\Vert \pi_n \Vert \to 0  \quad \text{as} \quad n \to \infty$. Then the representation 

\begin{equation*}
M_T = B_0 + \int_{0}^{T} \textbf{1}_{ \{ B_t = M_t \} } dB_t
\end{equation*}

does not hold, where the integral in the right hand side is understood as limit of Riemann-Stieltjes sums over partitions $\pi_n$ almost surely.

\end{corl}

\begin{thm}

Let $B=\{ B_t \}_{t \in [0,T]}$ be a fractional Brownian motion with Hurst parameter $H> \frac{1}{2}$. Then the integral

\begin{equation*}
\int_{0}^{T} \textbf{1}_{ \{ B_t = M_t \} } dB_t 
\end{equation*}

can be understood in the sense of generalized Lebesgue-Stieltjes integral and actually is equal to $0$ almost surely.

\end{thm}

\begin{proof} Let $f(t)= \textbf{1}_{ \{ B_t = M_t \} } $. Then according to theorem~\ref{thm:sections} we have

 \begin{equation*}
 \Vert f \Vert_{2,\beta} = 0 \quad \text{for every} \quad \beta \in (1-H,\frac{1}{2}).
 \end{equation*}

Now, the claim follows by proposition~\ref{pr:n-r} and theorem~\ref{t:n-r}.

\end{proof}

\begin{corl}

Let $B=\{ B_t \}_{t \in [0,T]}$ be a fractional Brownian motion with Hurst parameter $H> \frac{1}{2}$. Then the representation

\begin{equation*}
M_T = B_0 + \int_{0}^{T} \textbf{1}_{ \{ B_t = M_t \} } dB_t
\end{equation*}

does not hold, where the integral in the right hand side is understood as generalized Lebesgue-Stieltjes integral.

\end{corl}

\begin{rmk}
It is not clear whether it is possible to have some explicit representation for the maximum random variable of fractional Brownian motion analogously to theorem \ref{thm:maxbm}.
\end{rmk}

\end{document}